\documentclass{article}
\usepackage{url}
\usepackage{graphicx}
\graphicspath{{}}
\usepackage[tbtags]{amsmath}
\usepackage{amsfonts,amssymb}
\usepackage{mathabx}
\usepackage{mathtools}
\mathtoolsset{showonlyrefs=fasle,showmanualtags}
\usepackage{amsthm} 
\theoremstyle{plain}
\newtheorem{theorem}{Theorem}
\newtheorem{lem}[theorem]{Lemma}
\newtheorem{prop}[theorem]{Proposition}

\theoremstyle{definition}
\newtheorem{definition}{Definition}

\theoremstyle{remark}
\newtheorem*{remark}{Remark}

\usepackage{amssymb}
\usepackage{algorithm}
\usepackage{algorithmic}
\usepackage{float}

\usepackage[english]{babel}
\usepackage{enumitem}
\usepackage{color,booktabs,colortbl}
\usepackage{latexsym,bm}
\usepackage{multicol}
\usepackage{comment}
\usepackage{geometry}
\usepackage{fancyhdr}
\usepackage{caption}
\usepackage{subcaption}
\newcommand{\norm}[1]{\left\lVert#1\right\rVert}

\renewcommand{\tt}{\ttfamily}

\DeclareMathOperator*{\argmin}{arg\,min}

\DeclareMathOperator*{\re}{Re}
\DeclareMathOperator*{\im}{Im}
\DeclareMathOperator*{\diag}{diag}

\DeclareMathOperator*{\dist}{dist}

\usepackage{listings}
\usepackage[backend=bibtex,citestyle=numeric-comp,sortcites,bibstyle=mybibstyle,firstinits=true,isbn=false,doi=false]{biblatex}
\DeclareFieldFormat{labelnumberwidth}{#1\addperiod}
\DeclareFieldFormat[article]{volume}{\bibstring{volume}\addnbspace\bfseries #1}
\DeclareFieldFormat[article]{number}{\bibstring{number}\addnbspace #1}
\renewbibmacro*{volume+number+eid}{%
  \printfield{volume}%
  \setunit{\addcomma\space}
  \printfield{number}%
  \setunit{\addcomma\space}%
  \printfield{eid}}
\AtEveryCitekey{%
  \ifentrytype{article}
    {\clearfield{title}}
    {}%
  \clearfield{doi}%
  \clearfield{url}%
  \clearlist{publisher}%
  \clearlist{location}%
  \clearfield{note}%
}

\DeclareFieldFormat{urldate}{\bibstring{urlseen}\space#1}
\renewbibmacro*{date}{
  \iffieldundef{year}
    {\ifentrytype{online}
       {\printtext[urldate]{\printurldate}}
       {\printtext[date]{\printdate}}}
    {\printfield[date]{year}}}

\DeclareFieldFormat*{title}{#1}
\DeclareFieldFormat[book]{date}{\textbf{#1}} 
\renewbibmacro*{title}{
  \ifboolexpr{ test {\iffieldundef{title}}
               and test {\iffieldundef{subtitle}} }
    {}
    {\ifboolexpr{ test {\ifhyperref}
                  and not test {\iffieldundef{doi}} }
       {\href{http://dx.doi.org/\thefield{doi}}
          {\printtext[title]{%
             \printfield[titlecase]{title}%
             \setunit{\subtitlepunct}%
             \printfield[titlecase]{subtitle}}}}
       {\ifboolexpr{ test {\ifhyperref}
                     and not test {\iffieldundef{url}} }
         {\href{\thefield{url}}
            {\printtext[title]{%
               \printfield[titlecase]{title}%
               \setunit{\subtitlepunct}%
               \printfield[titlecase]{subtitle}}}}
         {\printtext[title]{%
            \printfield[titlecase]{title}%
            \setunit{\subtitlepunct}%
            \printfield[titlecase]{subtitle}}}}%
     \newunit}%
  \printfield{titleaddon}%
  \clearfield{doi}%
  \clearfield{url}%
  \clearlist{language}
  \clearfield{note}
  \ifentrytype{article}
    {\adddot}
    {}}

\begin{document} \title{On Relaxed Averaged Alternating Reflections (RAAR) 
  Algorithm for Phase Retrieval from Structured Illuminations} \author{Ji
Li\thanks{liji597760593\@126.com,
LMAM, School of Mathematical Sciences, Peking University, Beijing
100871, China}\quad Tie
Zhou\thanks{tzhou\@math.pku.edu.cn,
LMAM, School of Mathematical Sciences, Peking University, Beijing
100871, China}}
\date{\today}
\maketitle
\begin{abstract} 
In this paper, as opposed to the random phase masks, the structured
illuminations with a pixel-dependent deterministic phase shift are
considered to derandomize the model setup. The RAAR algorithm is modified to adapt to two or more
diffraction patterns, and the modified RAAR algorithm operates in
Fourier domain rather than space domain. The local convergence
of the RAAR algorithm is proved by some eigenvalue analysis. Numerical simulations is presented to demonstrate the effectiveness and stability of the algorithm
compared to the HIO (Hybrid Input-Output) method. The numerical performances show the global convergence of the RAAR in our tests.
\end{abstract}
\section{Introduction}
\label{sec:some-lemas}

The phase retrieval problem arises in many engineering and science
applications, such as X-ray crystallography~\cite{Millane1990},
electron microscopy~\cite{Misell1973}, X-ray diffraction
imaging~\cite{Shechtman2015}, optics~\cite{Kuznetsova1988} and
astronomy~\cite{Fienup1982}, just name a few. In these applications,
one often has recorded the Fourier transform intensity of a complex
signal, while the phase information is infeasible. The recovery of the
signal from the intensity of its Fourier transform is called phase
retrieval. We refer the reader to the recent survey
papers~\cite{Shechtman2015,Jaganathan2015} for the recent progress on
this problem.

There are two fundamental issues that accompany the phase retrieval
problem. The first one is the non-uniqueness of the solution. Clearly, the solution to
phase retrieval have the following three ``trivial associates'': the
solutions up to a
unit magnitude complex coefficient, a shift in space-domain and a
conjugate reflection through the origin. Fortunately, those
solutions do not change the structure of the solution, and they are
accepted. References for
theoretical results in terms of uniqueness can be found in~\cite{Akutowicz1956,Akutowicz1957,Barakat1984}
for continuous setting and
in~\cite{Sanz1985,Hayes1982,Bates,Sanz1983}
for the discrete model. Those references point out that the solution is
almost \emph{relatively unique} in multidimensional cases, except for
the above three 
``trivial associates''.  In literature~\cite{Luke2002}, it is pointed
out that, these uniqueness results are of fundamental importance, but
these do not apply to numerical algorithms, in particular for the noisy
data. The second issue is how to design efficient numerical
algorithms. The most widely used methods are perhaps the error
reduction (ER) and its variants, such as HIO~\cite{Fienup1982},
HPR~\cite{Luke2003}, RAAR~\cite{Luke2004} and the difference
map. Since these methods involve the sequentially projection onto
the constant sets, they are called \emph{iterative projection methods}. Although, in principle, if the
Fourier magnitude measurements are sufficiently
oversampled~\cite{Miao2000}, the lost phase information can be
recovered, these iterative
projection methods easily stagnate at 
a local minimizer. Additional constraints such as real-valuedness and
nonnegativity does not always increase the probability of
finding a trivial associate solution. A unified evaluation of the iterative projection
algorithms for the phase retrieval can be found in
literature~\cite{Marchesini2007}. The above iteration schemes are the counterparts of the
corresponding 
iteration in the framework of convex set feasible
problem~\cite{Bauschke2002}. Take an example, the HIO method with
relaxation parameter $\beta=1$ is the Douglas-Rachford
algorithm. Since the HIO method with $\beta=1$ performs
best~\cite{Fienup1982}, we take the HIO and Douglas-Rachford as the
same without ambiguity. The phase retrieval involve the intensity
constraint set in Fourier space, which is nonconvex, there is no
theory to guarantee the convergence, unlike the convex setting.

The stagnation may be due to the nonuniqueness in the general sense,
i.e., except the trivial associates. They are two approaches to
restore the \emph{uniqueness} only up to a complex constant with unit
magnitude. A natural way is incorporating the structured illumination,
i.e., to collect the diffraction patterns of the
modulated object $\bm{w}(\bm{n})\bm{x}(\bm{n})$, where the waveforms
or patterns $\bm{w}(\bm{n})$ are known. The phase retrieval from structured
illuminations was formulated as a matrix completion problem, whose
convex relaxation is a convex trace-norm minimization
problem~\cite{Candes2013,Candes2014}. However, due to the lifting from 
vector to matrix, the approach is prohibitive for two-dimensional
problem. To overcome the memories consumption, a common least-squares optimization was proposed and solved by
gradient descent method from a special spectral initialization with
local convergence guarantee~\cite{Candes2015}. To ensure the
effectiveness of the lifting method and gradient descent method, the
patterns $\bm{w}(\bm{n})$ are assumed Gaussian or with admissible
distribution and a large number is needed. They can be considered as
the optimization approaches, the iterative projection
schemes can also  applied to the phase retrieval problem with
structured illumination. Structured
illumination method with random phase mask (in uniform distribution) is
proposed by Fannjiang~\cite{Fannjiang2012a}  to restore
the \emph{uniqueness}. In this setting, even the ER algorithm behaves
well for nonnegative image. And recently the local
geometrical convergence to a solution for HIO is
proved~\cite{Chen2015,Chen2015a}. An important difference between the
optimization approaches in~\cite{Candes2013,Candes2014,Candes2015} and the
standard iterative projection methods is that
their coded diffraction patterns are not oversampled. We emphasize that reducing the number of coded diffraction
patterns is crucial for the diffract-before-destruct approach and
oversampling is a small price to pay with current sensor
technology~\cite{Chen2015}.

However, the random phase mask is difficult to implement in
practice. To reduce the randomness of the random phase mask, we replace
the random phase mask by pixel-dependent deterministic phase
shift.  The physical realizable setup can be implemented in two
ways. If the diffraction is in the regime of Fraunhofer, the
phase shift can be result from optical instruments, such as optical
grating, ptychography and oblique illuminations (see~\cite{Candes2013}
and references therein). In the regime of Fresnel, the phase shift is
automatically introduced in the original signal/image. In this paper, we consider the
more stable RAAR algorithm instead of the HIO method used in~\cite{Chen2015}. A common and vexing problem for the HIO
method is that iterates will oscillation around the solution in the
noiseless case and even wander away from the neighborhood of the
solution in the presence of noise (see our numerical simulations in
Section~\ref{sec:numer-simul}). We modify the RAAR algorithm (still
called RAAR) to adapt to two or more diffraction
patterns, which operates in Fourier domain rather than space domain.

The rest of the paper is
organized as follows. In Section~\ref{sec:meaur-setup-some}, we clarify the oversampling
measurement scheme and provide some preliminaries. In Section~\ref{sec:raar-algorithm}, we
describe the RAAR algorithm for two or more diffraction patterns. The local convergence of RAAR is proved under some
assumptions in Section~\ref{sec:local-convergence}. In Section~\ref{sec:numer-simul}, we present numerical
examples and demonstrate numerical global convergence of the RAAR
algorithm. Section~\ref{sec:conclusion} conclude the results of the paper.

\section{Measurements setup and some preliminaries}
\label{sec:meaur-setup-some}

\subsection{Oversampled diffraction patterns}
\label{sec:overs-diffr-patt}

We consider the discrete version of the phase retrieval problem. Given
a vector with nonnegative integer elements
$\bm{n}=(n_1,\ldots,n_d)\in\mathbb{N}^d$ and a vector with complex components
$\bm{z}=(z_1,\ldots,z_d)\in\mathbb{C}^d$, we define the multi-index
notation $\bm{z}^{\bm{n}}=z_1^{n_1}z_2^{n_2}\cdots z_d^{n_d}$. Let $\mathcal{C}(\bm{n})$ denote the set of complex-valued sequences
on $\mathbb{N}^d$ vanishing outside
\begin{equation*}
  \mathcal{N}=\{\,\bm{0}\leq \bm{n}\leq \bm{N}\,\},\quad \bm{N}=(N_1,\ldots,N_d),
\end{equation*}
where $\bm{n}\geq\bm{0}$ if $n_j\geq 0,\forall j$, and the
cardinality of $\mathcal{N}$ is $\lvert\mathcal{N}\rvert=\prod_j
N_j$. Then
the $d$-dimensional $z$-transform of a sequence $x(\bm{n})\in\mathcal{C}(\bm{n})$ may be
written compactly as
\begin{equation*}
  X(\bm{z})=\sum_{\bm{n}}x(\bm{n})\bm{z}^{-\bm{n}}.
\end{equation*}
All sequences $x(\bm{n})$ are assumed to have $z$-transforms with a
region of convergence that includes the unit ball $\lvert
z_k\rvert=1,k=1,\ldots,d$, so that the Fourier transform may be obtained as
\begin{equation*}
  X(\bm{\omega})=\left. X(\bm{z})\right\rvert_{\bm{z}=e^{i2\pi\bm{\omega}}}=\sum_{\bm{n}}x(\bm{n})e^{-i2\pi\bm{\omega}\cdot
  \bm{n}},\quad \bm{\omega}\in\mathbb{R}^d. 
\end{equation*}
Written in polar form, $X(\bm{\omega})$ is represented in terms of its
magnitude and phase as
\begin{equation*}
  X(\bm{\omega})=\lvert X(\bm{\omega})\rvert \exp(i\phi_x(\bm{\omega})).
\end{equation*}

We have that
\begin{equation*}
  \lvert
  X(\bm{\omega})\rvert^2=\sum_{\bm{n}=-\bm{N}}^{\bm{N}}\sum_{\bm{m}+\bm{n}\in\mathcal{N}}x(\bm{m}+\bm{n})x^*(\bm{m})e^{-i2\pi\bm{\omega}\cdot
  \bm{n}},
\end{equation*}
where $*$ denotes the conjugate.
We see that the Fourier intensity measurement is equivalent to the
Fourier transform measurement of the autocorrelation function of
$x(\bm{n})$:
\begin{equation*}
  r(\bm{n})=x(\bm{n})\star x^*(-\bm{n})=\sum_{\bm{m}+\bm{n}\in\mathcal{N}}x(\bm{m}+\bm{n})x^*(\bm{m}),
\end{equation*}
where $\star$ denotes the convolution.
From the relation between Fourier intensity measurement and the
autocorrelation function, since the autocorrelation sequences
$r(\bm{n})$ is defined on the enlarged grid
$\mathcal{M}=\{\,-\bm{N}\leq\bm{n}\leq\bm{N}\,\}$, whose cardinality is roughly $2^d$ times of the cardinality of
$\mathcal{N}$. By dimension counting, we should sample the
Fourier transform magnitude with more $2^d$ points than the space
discrete grid points, it is equivalent to apply DFT on a new sequence
with padding the original sequence by zeros. It is the standard
oversampling method. Then we may recover the original sequence from
the autocorrelation function or the Fourier transform intensity.

\subsection{Some preliminaries}
\label{sec:some-premilinaries}

We define a conjugate symmetric polynomial as follows~\cite{Hayes1982,Pitts2003}.
\begin{definition}[Conjugate Symmetry]
A polynomial $X(\bm{z})$ in $\bm{z}^{-1}$ is said to be
conjugate symmetric if
\begin{equation*}
  X(\bm{z})=\pm \bm{z}^{-\bm{k}}X^*(1/\bm{z}^*),
\end{equation*}
 for some vector $\bm{k}$ of positive integers. 
\end{definition}

We define the conjugate space-reversed polynomial
\begin{equation*}
  \widetilde{X}(\bm{z})=\bm{z}^{-\bm{N}}X^*(1/\bm{z}^*).
\end{equation*}
Clearly, the functions $X(\bm{z})$ and $\widetilde{X}(\bm{z})$ are
both polynomial in $\bm{z}^{-1}$. To understand the meaning of the operation $X(\bm{z})\rightarrow
 \bm{z}^{-\bm{N}}X^*(1/\bm{z}^*)$, we consider the $z$-transform of a
 sequences $\{\,a_0,a_1,\ldots,a_N\,\}$
 \begin{equation*}
   X(z)=a_0+a_1z^{-1}+\cdots+a_Nz^{-N}.
 \end{equation*}
Then,
\begin{equation*}
  z^{-N}X^*(1/z^*)=a_N^*+a_{N-1}^*z^{-1}+\cdots+a_0^*z^{-N}
\end{equation*}
is the $z$-transform of $\{\,a_N^*,a_{N-1}^*,\ldots,a_0^*\,\}$, which
is the conjugate time (space)-reversed sequence. 

Using the convolution theorem, the $z$-transform of the
autocorrelation function $r(\bm{n})$ is given by
\begin{equation}
\label{eq:1}
  R(\bm{z})=X(\bm{z})X^*(1/\bm{z}^*).
\end{equation}
From the fundamental theorem of algebra, the polynomial $X(\bm{z})$ in
$\bm{z}^{-1}$ always can be written uniquely (up to a factors of zero
degree) as the Hadamard product
\begin{equation}
\label{eq:2}
  X(\bm{z})=\alpha\bm{z}^{-\bm{n}_0}\prod_{k=1}^p X_k(\bm{z}),
\end{equation}
where $\bm{n}_0$ is a vector of nonnegative integers, $\alpha$ is a
complex coefficient, and $X_k(\bm{z})$ are nontrivial irreducible
polynomials in $\bm{z}^{-1}$. Combining \eqref{eq:1} and \eqref{eq:2},
we have
\begin{equation}
\label{eq:3}
  R(\bm{z}) = \lvert \alpha\rvert^2\prod_{k=1}^p X_k(\bm{z})X_k^*(1/\bm{z}^*).
\end{equation}
Since the trivial factors $\bm{z}^{-\bm{n}_0}$ represent the shifts in
space domain and hence contain information about the location of the
signal, So the cancellation of the linear phase factors \eqref{eq:3}
destroys information about location in the space domain.

We consider the more convenient function $Q(\bm{z})$:
\begin{equation*}
  Q(\bm{z})=\bm{z}^{-\bm{N}}R(\bm{z})=\lvert\alpha\rvert^2\prod_{k=1}^pX_k(\bm{z})\widetilde{X}_k(\bm{z}).
\end{equation*}
Using the convolution theorem and conjugate symmetry, we see that the
irreducible factors of $Q(\bm{z})$ define a set of sequences in the
space domain whose convolution yields (to within a factor of zero
degree) the autocorrelation sequence $r(\bm{n})$. These sequences
corresponding to the irreducible factor are the basic building blocks
for $r(\bm{n})$. Thus, asking if a sequences possesses a reducible
$z$-transform is equivalent to asking if there are smaller sequences,
which may be convolved to yield the original. The conditions under
which it is possible to recover the original sequence from the
magnitude of its Fourier transform are simply those which allow
unambiguous separation of the irreducible factors of $Q(\bm{z})$ into
those belong to $x(\bm{n})$ and those corresponding to their tilde
counterparts.

Hayes~\cite{Hayes1982} provided the theorems in this section, but he
only consider the signal or sequence $x(\bm{n})$ is real-valued. For
the comprehensiveness and easy reference of the paper, we now offer a
simple extension of the theorems to complex case.
\begin{theorem}[$z$-transform factorization]
\label{thm:1}
Let the $z$-transform $X(\bm{z})$ of a sequence
$x(\bm{n})\in\mathcal{C}(\bm{n})$ be given by
\begin{equation*}
  X(\bm{z})=\alpha\bm{z}^{-\bm{n}_0}\prod_{k=1}^pX_k(\bm{z}),\quad
  \bm{n}_0\in \mathbb{N}^d,\quad \alpha\in\mathbb{C},
\end{equation*}
where $X_k(\bm{z}),k=1,\ldots,p$ are nontrivial irreducible
polynomials. Let $Y(\bm{z})$ be the $z$-transform of another sequence
$y(\bm{n})\in\mathcal{C}(\bm{n})$. Suppose $\lvert
X(\bm{\omega})\rvert=\lvert Y(\bm{\omega})\rvert$ for all
$\bm{\omega}$, then $Y(\bm{z})$ must have the form
\begin{equation*}
  Y(\bm{z})=\lvert \alpha\rvert e^{i\theta}\bm{z}^{-\bm{m}}\prod_{k\in
  I_1}X_k(\bm{z})\prod_{k\in I_2}\widetilde{X}_k(\bm{z}),\quad
\bm{m}\in\mathbb{N}^d,\quad \theta\in\mathbb{R},
\end{equation*}
where $I_1$ and $I_2$ are complementary subsets of the integers in the
set $\{\,1,2,\ldots,p\,\}$.
\end{theorem}
\begin{proof}
  From the condition $\lvert
X(\bm{\omega})\rvert=\lvert Y(\bm{\omega})\rvert$, the $z$-transform
of the autocorrelation functions of the sequences $x(\bm{n})$ and
$y(\bm{n})$ are equal, i.e.,
\begin{equation*}
  X(\bm{z})X^*(1/\bm{z}^*) = Y(\bm{z})Y^*(1/\bm{z}^*).
\end{equation*}
Let $y(\bm{n})$ have a $z$-transform given by
\begin{equation*}
  Y(\bm{z})=\beta\bm{z}^{-\bm{m}_0}\prod_{k=1}^qY_k(\bm{z}).
\end{equation*}
Thus,
\begin{equation}
\label{eq:5}
  \lvert\alpha\rvert^2\prod_{k=1}^pX_k(\bm{z})X_k^*(1/\bm{z}^*)=\lvert\beta\rvert^2\prod_{k=1}^qY_k(\bm{z})Y_k^*(1/\bm{z}^*).
\end{equation}
Multiplying both sides of \eqref{eq:5} by $\bm{z}^{-\bm{N}}$ gives
\begin{equation}
  \label{eq:6}
  \lvert\alpha\rvert^2\bm{z}^{-\bm{m}_1}\prod_{k=1}^pX_k(\bm{z})\widetilde{X}_k(\bm{z})=\lvert\beta\rvert^2\bm{z}^{-\bm{m}_2}\prod_{k=1}^qY_k(\bm{z})\widetilde{Y}_k(\bm{z}).
\end{equation}
All factors on both sides of \eqref{eq:6} are polynomials in
$\bm{z}^{-1}$ and $\bm{m}_1\geq\bm{0},\bm{m}_2\geq \bm{0}$. The unique
factorization theorem implies $\bm{m}_1=\bm{m}_2$, $p=q$. Thus
$Y(\bm{z})$ is of the form
\begin{equation*}
  Y(\bm{z})=\eta\bm{z}^{-\bm{m}}\prod_{k\in I_1}X_k(\bm{z})\prod_{k\in
  I_2}\widetilde{X}_k(\bm{z}).
\end{equation*}
From the magnitude of Fourier transform are the same, so
$\eta=\lvert\alpha\rvert e^{i\theta}$.
\end{proof}

From the Theorem~\ref{thm:1}, we lack information about the linear phase terms
in \eqref{eq:6} and the unit magnitude complex $\theta$. These,
however, do not affect the shape of the recovered signal. If $X(\bm{z})$ has a single irreducible factor, we may
obtain $x(\bm{n})$ up to a complex unit magnitude constant, a shift
and a conjugate reflection through the origin. If we replace a
conjugate symmetric factor with its conjugate space-reversed
counterpart, we may just change the $x(\bm{n})$ by a sign~\cite{Pitts2003}.

\begin{definition}[Equivalence]
  We say that $y(\bm{n})$ is equivalent to $x(\bm{n})$ if
  \begin{equation*}
    y(\bm{n})=
    \begin{cases}
      e^{i\theta}x(\bm{k}+\bm{n}),\\
      e^{i\theta}x^*(\bm{k}-\bm{n}).
    \end{cases}
  \end{equation*}
for some real scalar $\theta$ and some vector $\bm{k}$ with integer components.
We denote as $y\sim x$.
\end{definition}
\begin{theorem}[Uniqueness of phase retrieval]
\label{thm:2}
  Let $x(\bm{n})\in\mathcal{C}(\bm{n})$ have a $z$-transform with at
  most one irreducible nonconjugate symmetric factors, i.e.,
  \begin{equation*}
    X(\bm{z})=P(\bm{z})\prod_{k=1}^pX_k(\bm{z}),
  \end{equation*}
where $P(\bm{z})$ is irreducible and $X_k(\bm{z})$ are irreducible and
conjugate symmetric. If $y(\bm{n})\in\mathcal{C}(\bm{n})$ with $\lvert
X(\bm{\omega})\rvert=\lvert Y(\bm{\omega})\rvert$, then $y\sim
x$.
\end{theorem}
Though, we can obtain at most $2^{(p-1)}$ different signals with
the same Fourier transform magnitude in one-dimensional phase
retrieval, the solutions for phase retrieval in two or more
dimensional are almost unique in the equivalent sense. This is due to
the fact that ``almost all'' polynomials in two or more variables are
irreducible~\cite{Hayes1982a}. It has also been shown for the case of polynomials with
real coefficients that the geometric character of the set of reducible
polynomials provides a stable framework for the retrieval of the phase magnitude~\cite{Sanz1983}.
\begin{theorem}[Uniqueness of magnitude retrieval]
\label{thm:3}
Let $x(\bm{n}),y(\bm{n})\in\mathcal{C}(\bm{n})$. If $X(\bm{z})$ and
$Y(\bm{z})$ have no nontrivial symmetric factors, i.e., trivial linear
phase factors are excluded. If $\tan(\phi_x(\bm{\omega}))=\tan(\phi_y(\bm{\omega}))$ for all $\bm{\omega}$, then $x(\bm{n})=\beta
y(\bm{n})$ for some real number $\beta$.
\end{theorem}
\begin{proof}
  Consider the sequence $g(\bm{n})$
  \begin{equation*}
    g(\bm{n})=x(\bm{n})\star y(\bm{n}),
  \end{equation*}
whose $z$-transform is given by
\begin{equation*}
  G(\bm{z})=X(\bm{z})Y^*(1/\bm{z}^*).
\end{equation*}
Since the phase of the Fourier transform of $g(\bm{n})$ satisfies
\begin{equation*}
  \tan(\phi_g(\bm{\omega}))=0.
\end{equation*}
So it follows that $G(\bm{\omega})$ is real-valued. By analytic
continuation, we have
\begin{equation*}
  G(\bm{z})=G^*(1/\bm{z}^*)
\end{equation*} and
so
\begin{equation*}
  X(\bm{z})Y^*(1/\bm{z}^*)=X^*(1/\bm{z}^*)Y(\bm{z}).
\end{equation*}
Multiplying both sides of above equality by $\bm{z}^{-\bm{N}}$ results
in the following polynomial equation in $\bm{z}^{-1}$,
\begin{equation*}
  X(\bm{z})\widetilde{Y}(\bm{z})z^{-\bm{m}}=\widetilde{X}(\bm{z})Y(\bm{z})\bm{z}^{-\bm{n}},
\end{equation*}
where $\bm{m}$ and $\bm{n}$ are integer-valued vectors with
$\bm{m}\geq \bm{0}$ and $\bm{n}\geq\bm{0}$. Now consider an arbitrary
nontrivial irreducible factor $X_k(\bm{z})$ of $X(\bm{z})$. If
$X_k(\bm{z})$ is associated with a factor of $\widetilde{X}(\bm{z})$,
then
\begin{equation*}
  X_k(\bm{z})=\alpha \widetilde{X}_i(\bm{z})
\end{equation*}
for some $i$. If $i=k$, then
$X_k(\bm{z})=\alpha^2X_k(\bm{z})$. Therefore, $\alpha=\pm 1$ and
$X_k(\bm{z})$ is symmetric. If $i\neq k$, then
\begin{equation*}
  X_k(\bm{z})X_i(\bm{z})=\alpha\widetilde{X}_i(\bm{z})X_i(\bm{z})
\end{equation*}
and $\widetilde{X}_i(\bm{z})X_i(\bm{z})$ is a symmetric factor of
$X(\bm{z})$. Consequently, each nontrivial irreducible factor of
$X(\bm{z})$ must be associated with a factor of $Y(\bm{z})$. By the
same argument, each nontrivial irreducible factor of $Y(\bm{z})$ must
be associated with a factor of $X(\bm{z})$. Therefore, $X(\bm{z})$ and
$Y(\bm{z})$ may differ by at most a trivial factor, i.e.,
\begin{equation*}
  Y(\bm{z})=\beta\bm{z}^{\bm{k}}X(\bm{z}).
\end{equation*}
However, if the tangent of the phase of $x(\bm{n})$ and $y(\bm{n})$
are equal, then $\bm{k}=\bm{0}$. It completes the proof.
\end{proof}

\subsection{Structured illumination}
\label{sec:struct-illum}

We will focus on a special case of structured illumination. A phase
shift dependent on the location of the signal in space domain is added
to the signal before diffraction. If the original signal is denoted by
$x(\bm{n})$, then the actual transformed signal is 
\begin{equation*}
  \tilde{x}(\bm{n})=x(\bm{n})\exp\left(\frac{ik}{2l}\bm{n}\cdot\bm{n}\right).
\end{equation*}
If the diffraction is in the regime of Fraunhofer, the
phase shift can be result from optical instruments, such as optical
grating, ptychography and oblique illuminations (see~\cite{Candes2013}
and references therein), Some application
of those structured illumination can be found
in~\cite{zhang2015spread,Williams2010,liu2008phase,johnson2008coherent}. If we record the
magnitude of the diffraction patterns in the Fresnel regime, the phase shift is
automatically introduced in the original signal/image. In the latter
model setup, the vector $\bm{n}$ is 
two-dimensional and typically represents coordinates transverse to a
coordinate axis and scalar $k=2\pi/\lambda$ is the wavenumber, and
$\lambda$ is the illumination wavelength. For a complex signal, we consider the measurements
with different distance $l$s. This kind of multiple measurement approaches are referred as phase diversity in astronomy~\cite{Gonsalves2014,Fienup1998}.

The advantage of the structured illumination is that the uniqueness
up to a unit complex constant is generally guaranteed. For example,
assume that the signal is real-valued and the $z$-transform of the product
signal $\tilde{x}(\bm{n})$ has at most one irreducible nonconjugate symmetric
factor, then its equivalences should be of the form:
\begin{equation*}
  y(\bm{n})=
  \begin{cases}
    e^{i\theta}x(\bm{m}+\bm{n})\exp\left(\frac{ik}{2l}(\bm{m}+\bm{n})\cdot
      (\bm{m}+\bm{n})\right)/\exp(\frac{ik}{2l}\bm{n}\cdot\bm{n}),\\
    e^{i\theta}x^*(\bm{m}-\bm{n})\exp\left(\frac{-ik}{2l}(\bm{m}-\bm{n})\cdot
      (\bm{m}-\bm{n})\right)/\exp(\frac{ik}{2l}\bm{n}\cdot\bm{n}).
  \end{cases}
\end{equation*}
Since the signal is real-valued, we have that $y(\bm{n})=\pm
x(\bm{n})$. If the signal is complex, then two phase-shift oversampled
diffraction patterns along two distances $l_1,l_2$ guarantee the signal
$y(\bm{n})=e^{i\theta}x(\bm{n})$ for some real constant $\theta$. 

We consider the two-dimensional phase retrieval, the propagation
matrix with one phase shift diffraction pattern is given by the matrix (operator)
\begin{equation}
\label{eq:20}
  A(x(\bm{n})) =\Phi(x(\bm{n})\exp(id\bm{n}\cdot \bm{n})),
\end{equation}
where $\Phi$ is the oversampled two-dimensional discrete Fourier
transform (DFT). More specifically $\Phi\in\mathbb{C}^{\lvert
  \mathcal{M}\rvert\times \lvert\mathcal{N}\rvert}$ is the sub-column
matrix of the standard DFT on the
extended grid $\mathcal{M}$.

For the two diffraction pattern case, the propagation matrix
(operator) is the stacked DFTs, i.e.,
\begin{equation}
\label{eq:21}
  A(x(\bm{n}))=
  \begin{bmatrix}
    \Phi(x(\bm{n})\exp(id_1\bm{n}\cdot \bm{n}))\\
    \Phi(x(\bm{n})\exp(id_2\bm{n}\cdot \bm{n}))
  \end{bmatrix}.
\end{equation}

\section{Local Convergence of RAAR algorithm}
\label{sec:raar-algorithm}

\subsection{Notation}
\label{sec:raar-algorithm-1}

For
a two-dimensional signal (sequences/image)
$\bm{x}\in\mathbb{C}^{n_1\times n_2}$, we will denote it as
$\bm{x}\in\mathbb{C}^n,n=n_1\times n_2$ by vectorizing the matrix. We
consider the phase retrieval problem in two case: real and complex
case. Let $\mathcal{X}$ be a nonempty closed convex set in
$\mathbb{C}^n$ and 
\begin{equation*}
  [\bm{x}]_{\mathcal{X}}=\argmin_{\bm{x}'\in \mathcal{X}}\norm{\bm{x}'-\bm{x}}
\end{equation*}
is the projection onto $\mathcal{X}$. Sine the DFT operator $\Phi$ can
be represented by a DFT matrix, so the propagation operator can be
written as matrix $A\in\mathbb{C}^{m\times n}$, where $A$ is
isometric, that is $A^*A=I$. The data is $\bm{b}=\lvert
A\bm{x}\rvert\in\mathbb{R}^m$. We focus on two cases: 
\begin{itemize}
\item [(a)]\emph{One-pattern case}: $A$ is given by \eqref{eq:20}, $\mathcal{X}=\mathbb{R}^n$, and
$m=2n$, 
\item [(b)]\emph{Two-pattern case}: $A$ is given by \eqref{eq:21},
$\mathcal{X}=\mathbb{C}^n$, and $m=4n$.
\end{itemize}

\subsection{RAAR algorithm in Fourier domain}
\label{sec:raar-algorithm-2}

Phase retrieval can be formulated as the following feasibility problem
in the Fourier domain
\begin{equation}
  \label{eq:4}
  \text{find}\quad \hat{\bm{y}}\in A\mathcal{X}\cap \mathcal{Y},
\end{equation}
where $\mathcal{Y}$ is the set which satisfies the Fourier domain
constraint, i.e., $\mathcal{Y}=\{\,\bm{y}\in\mathbb{C}^m\mid \lvert
\bm{y}\rvert=\bm{b}\,\}$. Let $P_1$ be the projection onto
$A\mathcal{X}$ and $P_2$ the projection onto $\mathcal{Y}$. For the
phase retrieval problem, the projections are of the following form:
\begin{equation}
  \label{eq:7}
  P_1\bm{y}=A[A^*\bm{y}]_{\mathcal{X}},\quad P_2\bm{y}=\bm{b}\circ
  \frac{\bm{y}}{\lvert \bm{y}\rvert},\quad \bm{y}\in\mathbb{C}^m,
\end{equation}
where $\circ$ is the Hadarmard product, which operates element-wise
product of the vectors. When $\lvert y\rvert=0$, the phase can be
assigned arbitrarily and we set $y/\lvert y\rvert=0$.

The RAAR method in Fourier
domain is given by the following iteration:
\begin{equation}
  \label{eq:8}
  \bm{y}_{k+1}=\beta T(\bm{y}_{k})+(1-\beta)P_2(\bm{y}_k),
\end{equation}
where
\begin{equation}
  \label{eq:9}
  T(\bm{y}_{k})=\frac{1}{2}(R_1R_2+I)(\bm{y}_{k})= [I+P_1(2P_2-I)](\bm{y}_k)
\end{equation}
 is the Douglas-Rachford iteration and $R_i=2P_i-I$ is the reflection
 operators.

The RAAR iteration can be written as the following form:
\begin{align}
  \label{eq:10}
  \bm{y}_{k+1}&=[\beta(I+P_1(2P_2-I))-(2\beta-1)P_2](\bm{y}_k)\\
&=\beta \Bigl(\bm{y}_k+A\left[A^*\left(2\bm{b}\circ\frac{\bm{y}_k}{\lvert\bm{y}_k\rvert}-\bm{y}_k\right)\right]_{\mathcal{X}}-\frac{2\beta-1}{\beta}\bm{b}\circ\frac{\bm{y}_k}{\lvert\bm{y}_k\rvert}\Bigr).
\end{align}

Note that the RAAR algorithm with $\beta=0.5$ becomes ER (Error Reduction) algorithm in
Fourier domain, since $P_1$ is linear for
$\mathcal{X}=\mathbb{R}^n$. When $\beta=1$, the unrelaxed RAAR
algorithm becomes HIO algorithm in
Fourier domain~\cite{Chen2015}. Since each of the algorithms involve the same basic
operations at each iteration, the rate of convergence of one algorithm
relative to another boils down to iteration counts. As we will demonstrate in numerical
simulations, ER algorithm stagnates at a local minimizer, while the HIO has the tendency to
escape the local minimizes~\cite{Marchesini2007}. Stability is of its
importance in numerical algorithm, which refers to the property that
the algorithm reliably approaches a neighborhood of a solution and
remains here, especially in the presence of noise. A common and
vexing problem for the HIO algorithm is that iterates will wander away from the neighborhood of a solution. The RAAR algorithm
is a trade-off between the ER algorithm and HIO algorithm, whose performance 
depends on the relaxation parameter $\beta$. Generally, RAAR algorithm
avoids the oscillation and instability of the HIO algorithm.
\subsection{Local convergence}
\label{sec:local-convergence}

In the following, we assume that the data $\bm{b}\neq
\bm{0}$ and $\lvert A\bm{x}\rvert\neq\bm{0}$ at the neighborhood of
the solution to phase retrieval. We consider the operator
\begin{equation*}
  K(\bm{y})=\bm{y}+AA^*\left(2\bm{b}\circ\frac{\bm{y}}{\lvert \bm{y}\rvert}-\bm{y}\right)-\gamma\bm{b}\circ\frac{\bm{y}}{\lvert \bm{y}\rvert}.
\end{equation*} 
According to the definition of G\^ateaux derivative, it follows that
\begin{align*} K(\bm{y}+\epsilon\bm{\eta})-K(\bm{y})&=\epsilon(I-AA^*)\bm{\eta}+(2 AA^*-\gamma I)\left(\frac{\bm{y}+\epsilon\bm{\eta}}{\lvert\bm{y}+\epsilon\bm{\eta}\rvert}-\frac{\bm{y}}{\lvert\bm{y}\rvert}\right)\circ \bm{b}\\
&=\epsilon(I-AA^*)\bm{\eta}+i\epsilon(2 AA^*-\gamma I)\diag\left(\frac{\bm{y}}{\lvert\bm{y}\rvert}\right)\im\left(\frac{\overline{\bm{y}}\circ\bm{\eta}}{\lvert\bm{y}\rvert^2}\right)\circ\bm{b},
\end{align*}
where we use the equality ($\lvert z\rvert \neq 0$),
\begin{equation*}
  D\left(\frac{z}{\lvert z\rvert}\right)(h)=\frac{h}{\lvert z\rvert}-\frac{z\re(\bar{z}h)}{\lvert z\rvert^3}.
\end{equation*}
We denote
\begin{equation*}
  \Omega =\diag\left(\frac{\bm{y}}{\lvert\bm{y}\rvert}\right),\quad B=\Omega^*A,
\end{equation*}
then it yields
\begin{align*}
  K(\bm{y}+\epsilon\bm{\eta})-K(\bm{y})&=\epsilon\Omega(I-BB^*)\Omega^*\bm{\eta}+i\epsilon\Omega(2 BB^*-\gamma I)\diag\left(\frac{\bm{b}}{\lvert \bm{y}\rvert}\right)\im(\Omega^*\bm{\eta})+o(\epsilon)\\
&=\epsilon\Omega J(\bm{v})+o(\epsilon),
\end{align*}
where
\begin{equation*}
  J(\bm{v})=(I-BB^*)\bm{v}+i(2 BB^*-\gamma I)\diag\left(\frac{\bm{b}}{\lvert \bm{y}\rvert}\right)\im(\bm{v}),\quad \bm{v}=\Omega^*\bm{\eta}.
\end{equation*}
When $\lvert\bm{y}\rvert=\bm{b}$, we have
\begin{equation*}
  J(\bm{v})=(I-BB^*)\bm{v}+i(2 BB^*-\gamma I)\im(\bm{v}).
\end{equation*}

The main result is local, geometric convergence of the RAAR
algorithm. The proof, which is a mimic of the
same result in literature~\cite{Chen2015}, is given in next section.
\begin{theorem}
\label{thm:1}
  Let $\bm{x}_0\in\mathbb{C}^n$ is the solution of the phase retrieval
  problem and the propagation matrix $A\in\mathbb{C}^{m\times n}$ is
  isometric. Suppose $m\geq 2n$ and 
  \begin{equation}
    \label{eq:11}
    \max_{\bm{z}\in\mathbb{C}^n,\,\bm{z}\perp i\bm{x}_0}
    \norm{\bm{z}}^{-1}\norm{\im(B\bm{z})}<1,\quad B =
    \diag\left(\frac{\overline{A\bm{x}_0}}{\lvert A\bm{x}_0\rvert}\right)A.
  \end{equation}
Let $\bm{y}_k$ be an RAAR iteration sequence and
$\bm{x}_{k}=A^*\bm{y}_k,k=1,\ldots$. If $\bm{x}_1$ is sufficient close
to $\bm{x}_0$, then for some constant $\eta<1$,
\begin{equation}
  \label{eq:12}
  \dist(\bm{x}_k,\bm{x}_0)\leq\eta^{k-1}\dist(\bm{x}_1,\bm{x}_0),
\end{equation}
where $\dist(\bm{x}_k,\bm{x}_0)$ is the distance
\begin{equation}
  \label{eq:13}
  \dist(\bm{x}_k,\bm{x}_0)=\min_{c\in\mathbb{C},\lvert c\rvert=1}\norm{c\bm{x}_k-\bm{x}_0}.
\end{equation}
\end{theorem}

\section{Spectral gap and the proof of Theorem~\ref{thm:1}}
\label{sec:eigenvalue-analysis}

In this section, we provide the result that \eqref{eq:11} is satified
for almost all two-dimensional phase retrieval. We 
introduce the following matrix $\mathcal{B}$ and map $G$, which maps
a complex to its real and image parts.
\begin{equation*}
  \mathcal{B}=
  \begin{bmatrix}
    \re(B)&-\im(B)
  \end{bmatrix}\in\mathbb{R}^{m\times 2n}, \text{ and } G(\bm{z})=(\re(\bm{z}),\im(\bm{z}))^T.
\end{equation*}
where $B=\Omega_0^*A$. We have
\begin{equation*}
  G(B\bm{z})=
  \begin{bmatrix}
    \mathcal{B}G(\bm{z})\\
    \mathcal{B}G(-i\bm{z})
  \end{bmatrix}\in\mathbb{R}^{2m},\quad \bm{z}\in\mathbb{C}^n.
\end{equation*}
and we have
\begin{equation*}
  \mathcal{B}^T=
  \begin{bmatrix}
    \re(B^*)\\\im(B^*)
  \end{bmatrix}\in\mathbb{R}^{2n\times m}\quad,\norm{B^*\im(\bm{v})}=\norm{\mathcal{B}^T\im(\bm{v})}.
\end{equation*}

We first derive the
singular values of the corresponding
matrix $\mathcal{B}$.
Since, $B\bm{x}_0=\lvert \bm{y}_0\rvert=\lvert \bm{b}\rvert$, so
\begin{equation*}
  \mathcal{B}G(\bm{x}_0)=\lvert \bm{y}_0\rvert,\quad \mathcal{B}G(-i\bm{x}_0)=0.
\end{equation*}
Let $\mathcal{B}=U\Sigma V^*$, and $\sigma_1\geq
\sigma_2\geq\cdots\geq\sigma_{2n}$ be the singular values of matrix
$\mathcal{B}$, then $\sigma_1=1$,
$\bm{v}_1=G(\bm{x}_0),\bm{u}_1=\lvert \bm{y}_0\rvert$ and
$\sigma_{2n}=0,\bm{v}_{2n}=G(-i\bm{x}_0)$.
\begin{lem}
\begin{equation*}
  \sigma_2=\max_{\bm{w}\in\mathbb{R}^{2n},\bm{w}\perp
  \bm{v}_1}\frac{\norm{\mathcal{B}\bm{w}}}{\norm{\bm{w}}}=\max_{\bm{z}\in\mathbb{C}^n,\bm{z}\perp
    i\bm{x}_0}\frac{\norm{\im(B\bm{z})}}{\norm{\bm{z}}}.
\end{equation*}
\end{lem}
\begin{proof}
  Let $\bm{z}_0\perp i\bm{x}_0$ and
  $\sigma_2=\frac{\norm{\im(B\bm{z}_0)}}{\norm{\bm{z}_0}}$, then 
 $0=\re\langle \bm{z}_0,i\bm{x}_0\rangle=\langle
G(-i\bm{z}_0),G(\bm{x}_0\rangle=0$. Take $\bm{w}_0=G(-i\bm{z}_0)$, we have $\max_{\bm{z}\in\mathbb{C}^n,\bm{z}\perp
    i\bm{x}_0}\norm{\bm{z}}^{-1}\norm{\im(B\bm{z})}\leq \max_{\bm{w}\in\mathbb{R}^{2n},\bm{w}\perp
  \bm{v}_1}\norm{\bm{w}}^{-1}\norm{\mathcal{B}\bm{w}}$. The inverse
inequality can be obtained by the fact that 
$\im(B\bm{z})=\mathcal{B}G(-i\bm{z})$ and $\norm{G(-i\bm{z})}=\norm{\bm{z}}$. This
completes the proof.  
\end{proof}
 
\begin{lem}
Let $\sigma_1\geq
\sigma_2\geq\cdots\geq\sigma_{2n}$ are the singular values of $\mathcal{B}$, we have that
$\sigma_k^2+\sigma_{2n+1-k}^2=1$ and
$\bm{v}_{2n+1-k}=G(-iG^{-1}(\bm{v}_k))$ and $\bm{v}_k=G(iG^{-1}(\bm{v}_{2n+1-k}))$.
\end{lem}
\begin{proof}
  Since $B$ is isometric, we have $\norm{\bm{w}}=\norm{B\bm{w}}$,
  $\forall \bm{w}\in\mathbb{C}^n$.

We have that
\begin{equation*}
  \norm{B\bm{w}}^2=\norm{G(B\bm{w})}^2=\norm{\mathcal{B}G(\bm{w})}^2+\norm{\mathcal{B}G(-i\bm{w})}^2,
\end{equation*}
Since $\norm{\bm{w}}^2=\norm{G(\bm{w})}^2$, we have
\begin{equation}
\label{eq:888}
  \norm{G(\bm{w})}^2=\norm{\mathcal{B}G(\bm{w})}^2+\norm{\mathcal{B}G(-i\bm{w})}^2.
\end{equation}

We prove the lemma by induction. It is obvious for $k=1$. Recall the
Courant-Fischer theorem, which characterize the singular values, we
have
\begin{align*}
  \sigma_j &= \max_{\norm{\bm{w}}=1}\norm{\mathcal{B}\bm{w}}, \quad
  \bm{w}\perp \bm{v}_1,\ldots,\bm{v}_{j-1},\\
\sigma_{2n+1-j}&=\min_{\norm{\bm{w}'}=1}\norm{\mathcal{B}\bm{w}'}, \quad
  \bm{w}'\perp \bm{v}_{2n},\ldots,\bm{v}_{2n+2-j}.
\end{align*}
Hence, by~\eqref{eq:888}, it follows that
\begin{align*}
  \sigma_k^2=\max_{\norm{\bm{w}}=1}\norm{\mathcal{B}\bm{w}}^2=1-\min{\norm{\bm{w}'}=1}\norm{\mathcal{B}\bm{w}'}^2,\quad \bm{w}\perp\bm{v}_1,\ldots,\bm{v}_{k-1},
\end{align*}
where $\bm{w}'=G(-iG^{-1}(\bm{w}))$. The condition $\bm{w}\perp
\bm{v}_1,\ldots,\bm{v}_{j-1}$ implies $\bm{w}'\perp
\bm{v}_{2n},\ldots,\bm{v}_{2n+2-j}$, where $\bm{v}_{2n+1-k}=G(-iG^{-1}(\bm{v}_k))$.
\end{proof}

To this end, we have to show the assumption of Theorem~\ref{thm:1} is satisfied for our
measurement setup for two-dimensional phase retrieval
problem. Furthermore, we show
that this condition~\eqref{eq:11} is satisfied for almost all two-dimensional
signals with at least one oversampled coded
diffraction patterns. This excludes some two-dimensional signals whose
$z$-transform have nontrivial symmetric factors,
see Theorem~\ref{thm:3}. However, it should holds for the signals we
consider in practical. 

The fact that the second singular value $\sigma_2$ is strictly less than one
is the immediate consequence of the following result.
\begin{prop}
  Let $A$ be isometric and $B=\Omega_0^*A$. Then
  $\norm{\im(B\bm{z})}=1$ holds for some unit vector $\bm{z}$ if and
  only if 
  \begin{equation*}
    \re(\bm{a}_j^*\bm{z})\re(\bm{a}_j\bm{x}_0)+\im(\bm{a}_j^*\bm{z})\im(\bm{a}_j\bm{x}_0)=0,\forall j=1,\ldots,m,
  \end{equation*}
where $\bm{a}_j^*$ are the rows of $A$, or equivalently
\begin{equation*}
 \bm{\omega}=\pm i\bm{\omega}_0, \quad\bm{\omega}=\frac{A\bm{z}}{\lvert
    A\bm{z}\rvert}, \quad\bm{\omega}_0=\frac{A\bm{u}_0}{\lvert
    A\bm{u}_0\rvert},
\end{equation*}
where the $\pm$ sign may be element-wise-dependent.
\end{prop}
\begin{proof}
  By the isometry of $B$, we have
\begin{equation*}
  \norm{\im(B\bm{z})}^2\leq
                         \norm{B\bm{z}}^2=\norm{\bm{z}}^2.
\end{equation*}
And the
inequality becomes an equality if and only if 
\begin{equation*}
  \re(B\bm{z})=\re\left(\frac{\overline{A\bm{x}_0}}{\lvert
      A\bm{x}_0\rvert}\circ A\bm{z}\right)=\bm{0}.
\end{equation*}
So the arguments of $\bm{\omega}_0$ and $\bm{\omega}$ differ by
$\frac{\pi}{2}$. It completes the proof.
\end{proof}

Sine we have $\measuredangle \bm{\omega}_0=\measuredangle
\pm i\bm{\omega}$, we have $\tan(\measuredangle \bm{\omega}_0)=\tan(\measuredangle
\pm i\bm{\omega})$, so we have that
$\bm{z}=ic\bm{x}_0$ for a real constant $c$ by the uniqueness of
magnitude retrieval under the assumption of $z$-transform of
$\bm{x}\exp(\frac{ik}{2l}\bm{n}\cdot\bm{n})$ having nontrivial
conjugate symmetric factors. So
\begin{equation*}
  \norm{\im(B\bm{z})}=1, \norm{\bm{z}}=1\text{ iff } \bm{z}=\pm i\bm{u}_0/\norm{\bm{u}_0}.
\end{equation*}
and hence
\begin{equation*}
  \sigma_2=\max_{\bm{z}\in\mathbb{C}^n,\bm{z}\perp i\bm{u}_0}\norm{\bm{z}}^{-1}\norm{\im(B\bm{z})}<1.
\end{equation*}

\begin{proof}[Proof of Theorem~\ref{thm:1}]
Let the solution to the phase retrieval be $\bm{x}_0$, then we have
$\bm{y}_0=A\bm{x}_0$. And let $\bm{v}_k=\Omega_0^*(c_k\bm{y}_k-\bm{y}_0)$, where $c_k$ is the minima phase such that $\norm{c_k\bm{y}_k-\bm{y}_0}$ is the minimum, then we have
\begin{align*}
 \Omega_0^* (c_k\bm{y}_{k+1}-\bm{y}_0)&=\beta\Omega_0^*(K(\bm{y}_k)-K(\bm{y}_0))\\
&=\beta \left(J(\bm{v}_k)+o(\norm{\bm{v}_k})\right).
\end{align*}
Moreover, multiplying $B^*=A^*\Omega_0$, it follows that
\begin{align*}
  c_k\bm{x}_{k+1}-\bm{x}_0=B^*\Omega_0^* (c_k\bm{y}_{k+1}-\bm{y}_0)&=\beta\left(B^*J(\bm{v}_k)+o(\norm{c_k\bm{x}_k-\bm{x}_0})\right)\\
&=\beta\left(B^*(I-BB^*)\bm{v}_k+i(2 B^*BB^*-\gamma B^*)\im(\bm{v}_k)\right)\\
&=iB^*\im(\bm{v}_k)
\end{align*}
by the isometric property $B^*B=I$ and $\gamma=(2\beta-1)/\beta$.

By the optimal phase $c_k$, we have
\begin{equation*}
  \re\langle\bm{v}_k, i\lvert \bm{y}_0\rvert\rangle=\re\langle c_k\bm{y}_k-\bm{y}_0,i\bm{y}_0\rangle=0, 
\end{equation*}
so $\im(\bm{v}_k)$ is orthogonal to the leading right singular vector
$\lvert\bm{y}_0 \rvert$ of $\mathcal{B}$.

So 
\begin{align*} \norm{c_k\bm{x}_{k+1}-\bm{x}_0}&=\norm{B^*\im(\bm{v}_k)}+o(\norm{c_k\bm{x}_k-\bm{x}_0})\\
&=\norm{\mathcal{B}^T\im(\bm{v}_k)}+o(\norm{c_k\bm{x}_k-\bm{x}_0})\\
&\leq \sigma_2\norm{\im(\bm{v}_k)}+o(\norm{c_k\bm{x}_k-\bm{x}_0})\\
&\leq \eta\norm{\bm{v}_k},
\end{align*}
where we use the fact that $\sigma_2<1$, so we have $\eta<1$.
\end{proof}
\begin{remark}
 It is can be seen that for ER algorithm, i.e., $\beta=0.5$, the local
 convergence holds. For HIO algorithm, i.e., $\beta=1$,
 the local convergence has been proved in~\cite{Chen2015}.
\end{remark}

From the proof, we found that one oversampled Fourier diffraction
pattern can ensure the local convergence, and the local convergence 
holds for ER and HIO algorithm, which are two special cases of RAAR
algorithm corresponding two specified $\beta$s. Note that the spectral
gap requires (i.e. $\sigma_2<1$) the setting is in two-dimensional.
\section{Numerical simulations}
\label{sec:numer-simul}
 We explore the performance of the RAAR algorithm in two cases,
 which differ by the number of the diffraction patterns recorded. For real image, 
one structured illumination ($d=3$) diffraction pattern is recorded.For
complex image, two structured illumination diffraction
($d_1=3,d_2=-3$) patterns are recorded. We
 test three real images and one complex image which have different sizes: (a)
cameraman with size $128\times 128$,  (b) lena with size $256\times 256$,
and (c) mandril with size $512\times 512$, (d) gold balls with size $512\times
512$. The only complex image is obtained by adding a random phase to
the original magnitude image. Figure \ref{fig:1} displays the three real images and the
complex golden balls image. For complex image, we plot the magnitude
instead of its real and image parts respectively.
\begin{figure}
  \begin{subfigure}[t]{.24\textwidth}
    \centering
    \includegraphics[width=\textwidth]{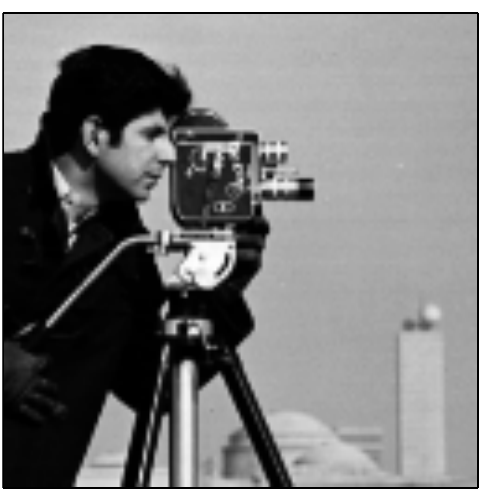}
\caption{cameraman\label{fig:1a}}
  \end{subfigure}
 \begin{subfigure}[t]{.24\textwidth}
    \centering
    \includegraphics[width=\textwidth]{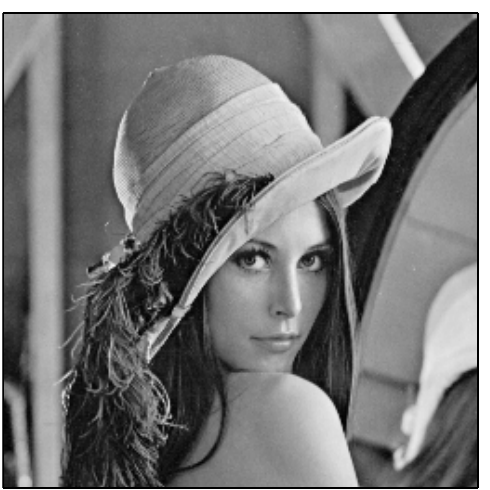}
\caption{lena\label{fig:1b}}
  \end{subfigure}
\begin{subfigure}[t]{.24\textwidth}
    \centering
    \includegraphics[width=\textwidth]{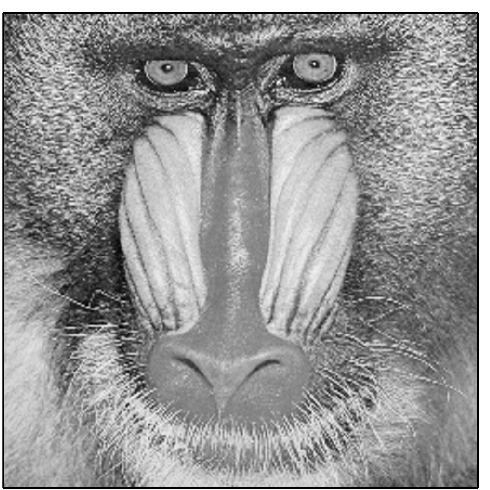}
\caption{mandril\label{fig:1c}}
  \end{subfigure}
\begin{subfigure}[t]{.24\textwidth}
    \centering
    \includegraphics[width=\textwidth]{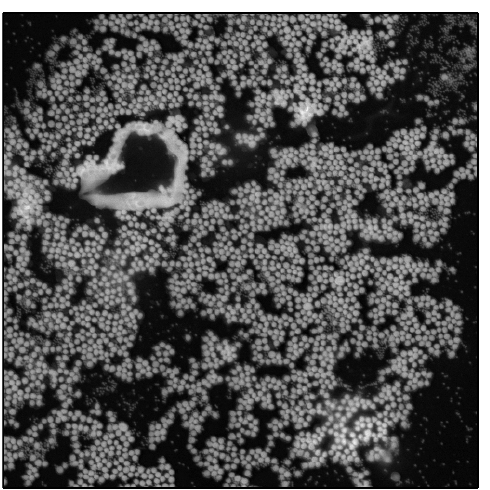}
\caption{gold balls\label{fig:1d}}
  \end{subfigure}
  \caption{The original test images}
  \label{fig:1}
\end{figure}

To compare the rate of convergence, we denote the relative error of
the iteration $\bm{x}_k$ as
\begin{equation*}
  \text{rel err} = \frac{\norm{c_k\bm{x}_k-\bm{x}_0}}{\norm{\bm{x}_0}},\quad c_k=\argmin_{c\in\mathbb{C}}\norm{c\bm{x}_k-\bm{x}_0}.
\end{equation*}
The maximum iteration is set $150$ for real images and $300$ for
complex images. And the initialization is the constant initialization,
where each pixel value is set to unity.
\subsection{Effect of the relaxed parameter $\beta$}
\label{sec:test-2:-parameters}
As we have said, the performance of the general RAAR algorithm depends
on the relaxed parameter $\beta$. If $\beta=0.5$, then it becomes the
ER algorithm, and $\beta=1$, it becomes the HIO algorithm or
Douglas-Rachford algorithm. We test three different $\beta$s: 0.8, 0.9
and 1. The performance of the RAAR algorithm for the four test images
are decipted in Figures~\ref{fig:2} for noiseless case. 
The performance of RAAR algorithm differ for real images and complex
images. For the three real images, the oscillation phenomenon with the HIO ($\beta=1$) can be observed
($\beta=0.9$ shows little oscillation for mandril) when
the iteration is in the neighborhood of the solution. Despite the
unwanted oscillation, the performance
of RAAR algorithm with $\beta=1$ is superior to that with $\beta=0.8$
and $0.9$. The performance comparison of $\beta=0.8$ and
$\beta=0.9$ does not show a uniform result, and it depends on the
problem we consider. For the small size image (a), the
case $\beta=0.8$ is superior to $0.9$. For image (b), $\beta=0.9$
is better than $0.8$. For (c), $0.8$ is better. For complex image (d),
the performance behavior is different. The RAAR algorithm with
$\beta=0.9$ is superior to HIO algorithm (i.e. $\beta=1$). 

We also test the RAAR algorithm when the diffraction data contain Poisson
noise. The performance of the RAAR algorithm with different $\beta$
is showed in Figure~\ref{fig:3} in the case of SNR level being $40$dB. It is
observed that in the presence of noise, the common instability of the
HIO algorithm is obvious. If it runs more long, the iteration keep
away the solution more far. This wandering of the iterations near an
local solution has been reported in~\cite{Luke2004}. The relaxations
in the RAAR algorithm can dampen the iterations near a solution. And the smaller the $\beta$ is, the
smaller the relative error for real images. When $\beta=0.5$, RAAR
algorithm becomes the ER algorithm, and ER algorithm can be considered
as a fixed stepsize gradient descent method for solving the
following optimization problem
\begin{equation*}
  \min_{\bm{x}\in\mathbb{C}^n} \quad f(\bm{x})=\norm{\lvert A\bm{x}\rvert-\bm{b}}^2.
\end{equation*} A gradient descent scheme with a fixed unity stepsize is
\begin{align*}
  \bm{x}_{k+1}&=\bm{x}_k-\nabla f(\bm{x}_k)\\
  &=\bm{x}_k-\left(\bm{x}_k-\left[A^*\left(\frac{A\bm{x}_k}{\lvert
    A\bm{x}_k\rvert}\circ \bm{b}\right)\right]_{\mathcal{X}}\right)\\
&=\left[A^*\left(\frac{A\bm{x}_k}{\lvert
    A\bm{x}_k\rvert}\circ \bm{b}\right)\right]_{\mathcal{X}}.
\end{align*}
It is the iteration of ER algorithm. Note that optimization approach is more stable than iterative
projection approaches. So small $\beta$ should be better, on the
contrary, the larger the $\beta$ is, the more
possible to escape the minima as HIO method. We  conclude that
$\beta=0.8$ is applicable to real images and $\beta=0.9$ for
complex images.
\begin{figure}
   \begin{subfigure}[b]{.5\textwidth}
\centering
\caption{}
\includegraphics{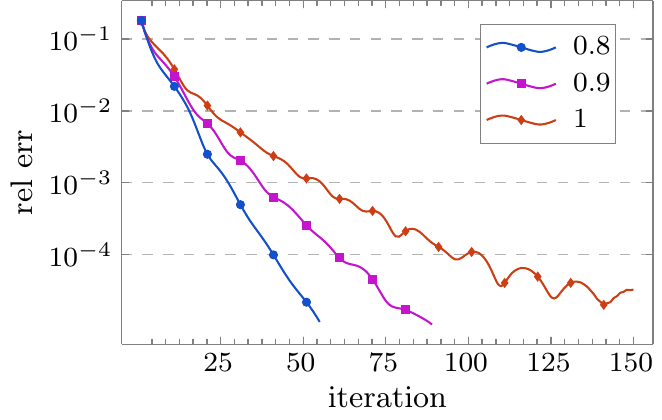}
   \end{subfigure}\hskip 30pt
   \begin{subfigure}[b]{.5\textwidth}
\centering
\caption{}
\includegraphics{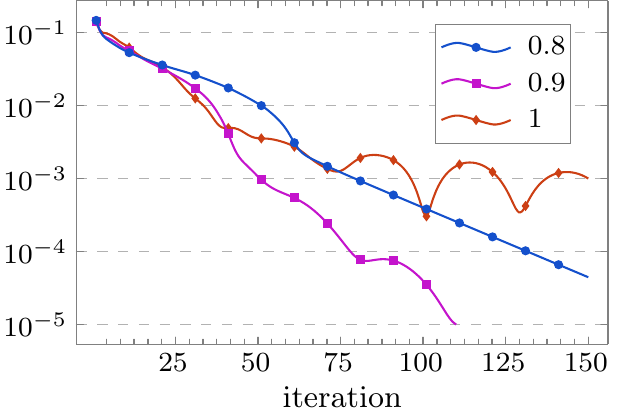}
   \end{subfigure}\\\medskip
   \begin{subfigure}[b]{.5\textwidth}
\centering
\caption{}
\includegraphics{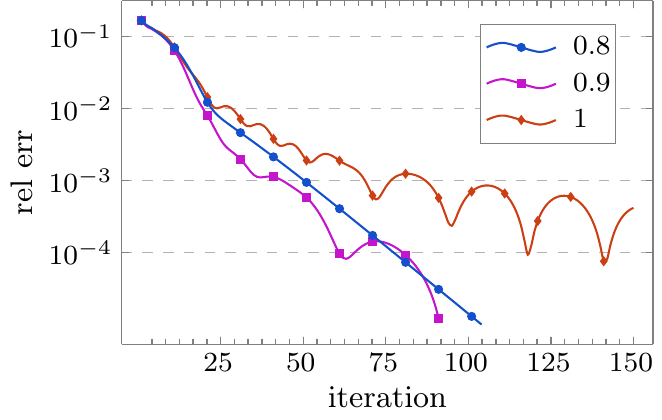}
   \end{subfigure} \hskip 30pt
     \begin{subfigure}[b]{.5\textwidth}
\centering
\caption{}
\includegraphics{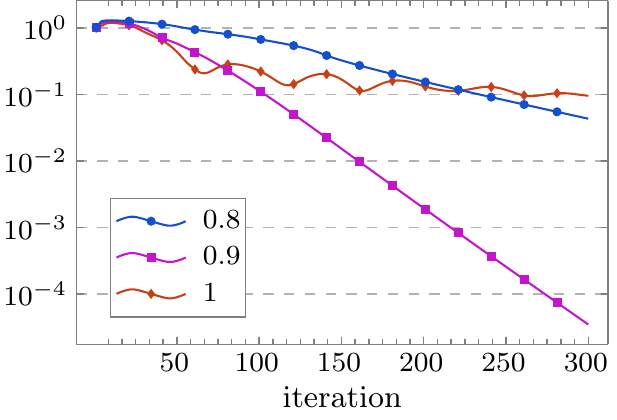}
   \end{subfigure}
  \caption{Relative error vs parameter $\beta$, no noise}
 \label{fig:2}
 \end{figure}

\begin{figure}
   \begin{subfigure}[b]{.5\textwidth}
\centering
\caption{}
\includegraphics{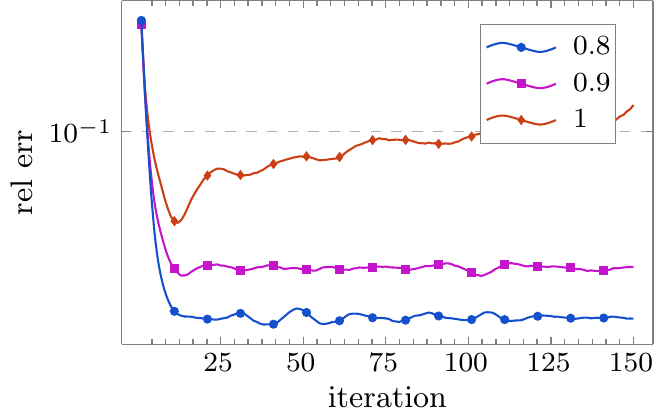}
   \end{subfigure}\hskip 30pt
   \begin{subfigure}[b]{.5\textwidth}
\centering
\caption{}
\includegraphics{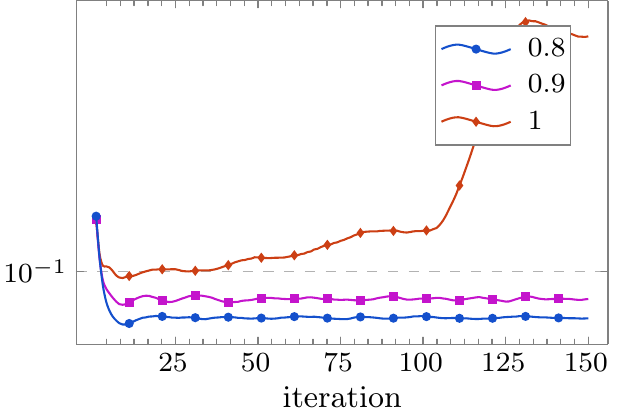}
   \end{subfigure}\\\medskip
   \begin{subfigure}[b]{.5\textwidth}
\centering
\caption{}
\includegraphics{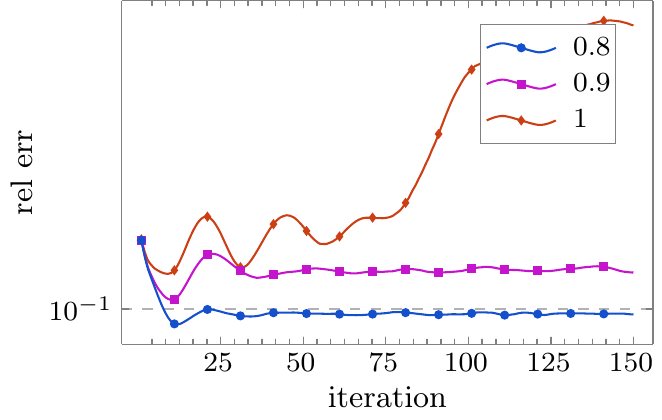}
   \end{subfigure} \hskip 30pt
     \begin{subfigure}[b]{.5\textwidth}
\centering
\caption{}
\includegraphics{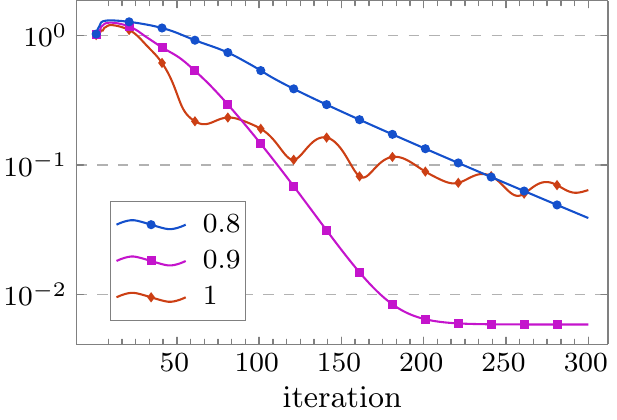}
   \end{subfigure}
 \caption{Relative error vs parameter $\beta$, SNR=$40$dB}
 \label{fig:3}
 \end{figure}

\subsection{Noisy measurements}
\label{sec:noisy-measurments}

In the second set of experiments we consider the same test images but
with noisy measurements. Since
the main noise yields Poisson distribution resulting from the photon
counting in practice, we add
random Poisson noise to the measurements for five different SNR levels,
ranging from 30dB to 50dB with step 5dB. Figure~\ref{fig:4} shows the average
relative error in dB versus the SNR\@. The error curve shows clearly
the linear behavior between SNR and relative error for complex
images (see Figure~\ref{fig:4b}). For real images, the curves imply
the instability of the phase retrieval from structured
illumination: the SNR becomes lower, the reconstruction becomes more
worse. To be clear, we show the
reconstructions from two SNR level data. Figures~\ref{fig:5} and~\ref{fig:6}
depict
the resulting reconstructions for SNR=$30$dB and SNR=$40$dB, respectively. 
\begin{figure}
\begin{subfigure}[b]{.5\textwidth}
  \centering
\includegraphics{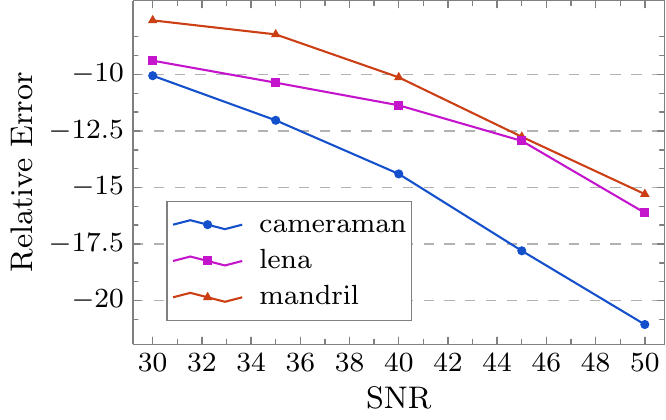}
\caption{Three real images\label{fig:4a}}
\end{subfigure}\hskip 20pt
\begin{subfigure}[b]{.5\textwidth}
  \centering
\includegraphics{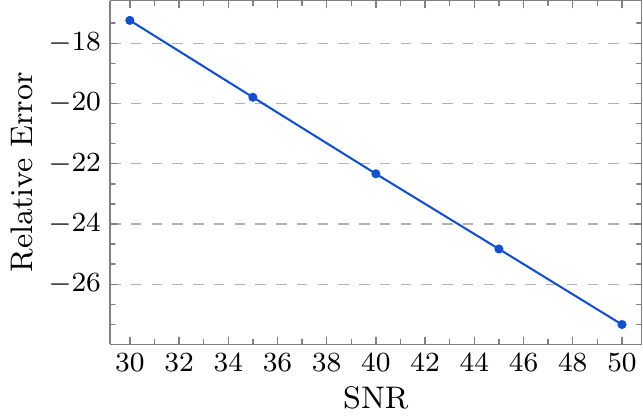}
\caption{Complex image\label{fig:4b}}
\end{subfigure}
\caption{Relative Error in dB vs SNR\label{fig:4}}
\end{figure}

\begin{figure}
  \begin{subfigure}[t]{.24\textwidth}
    \centering
    \includegraphics[width=\textwidth]{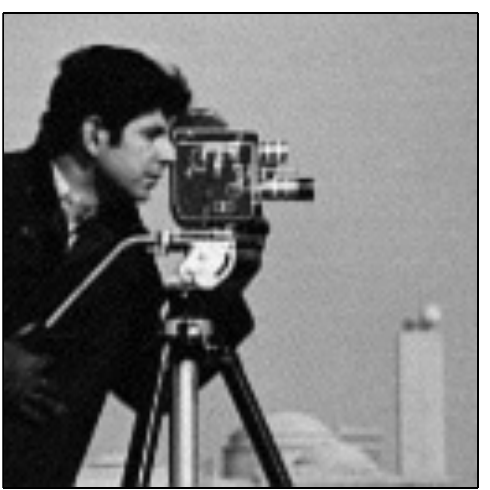}
\caption{cameraman\label{fig:5a}}
  \end{subfigure}
 \begin{subfigure}[t]{.24\textwidth}
    \centering
    \includegraphics[width=\textwidth]{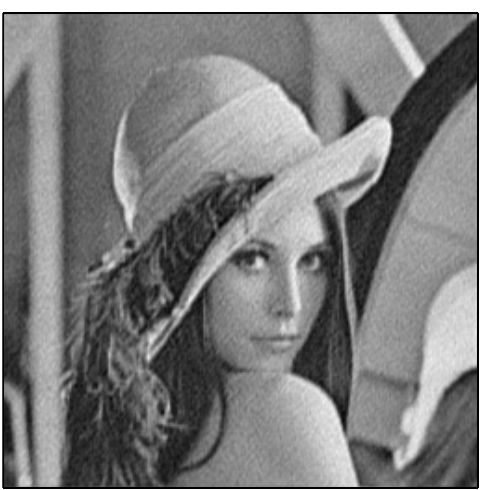}
\caption{lena\label{fig:5b}}
  \end{subfigure}
\begin{subfigure}[t]{.24\textwidth}
    \centering
    \includegraphics[width=\textwidth]{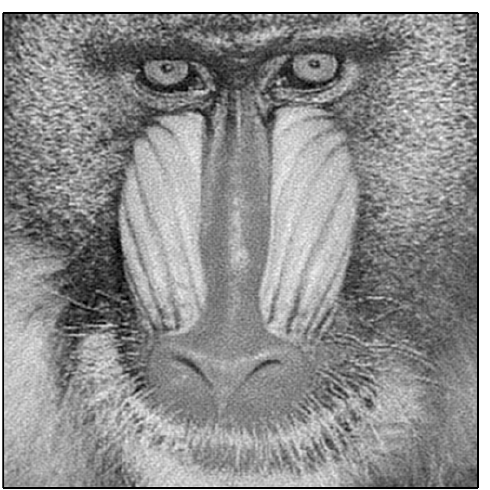}
\caption{mandril\label{fig:5c}}
  \end{subfigure}
\begin{subfigure}[t]{.24\textwidth}
    \centering
    \includegraphics[width=\textwidth]{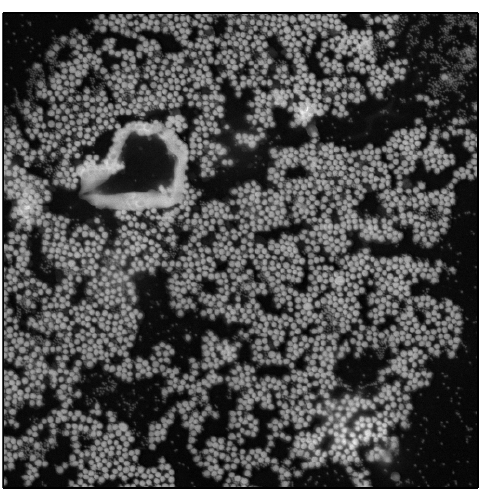}
\caption{gold balls\label{fig:5d}}
  \end{subfigure}
  \caption{Reconstruction, SNR = $40$dB}
  \label{fig:5}
\end{figure}

\begin{figure}
  \begin{subfigure}[t]{.24\textwidth}
    \centering
    \includegraphics[width=\textwidth]{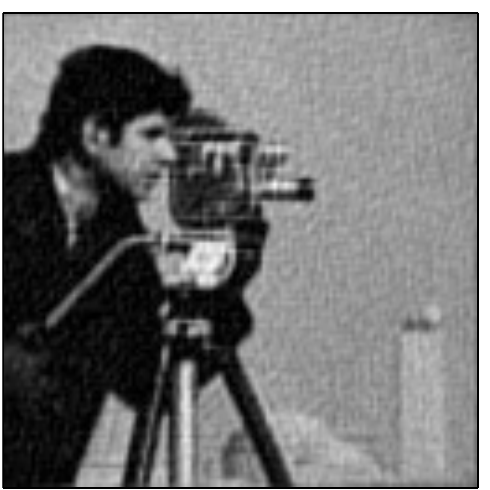}
\caption{cameraman\label{fig:6a}}
  \end{subfigure}
 \begin{subfigure}[t]{.24\textwidth}
    \centering
    \includegraphics[width=\textwidth]{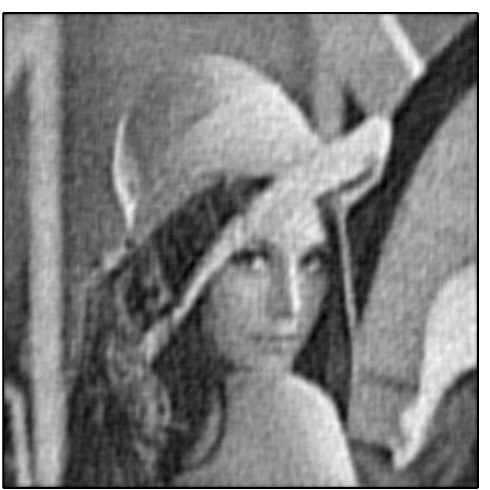}
\caption{lena\label{fig:6b}}
  \end{subfigure}
\begin{subfigure}[t]{.24\textwidth}
    \centering
    \includegraphics[width=\textwidth]{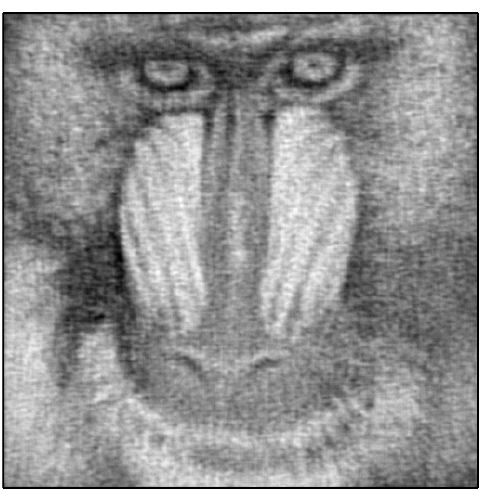}
\caption{mandril\label{fig:6c}}
  \end{subfigure}
\begin{subfigure}[t]{.24\textwidth}
    \centering
    \includegraphics[width=\textwidth]{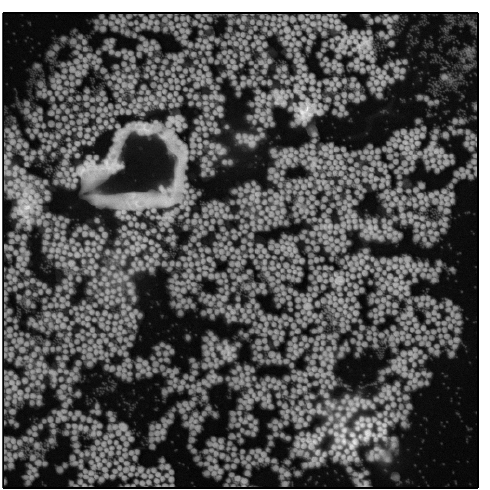}
\caption{gold balls\label{fig:6d}}
  \end{subfigure}
  \caption{Reconstruction, SNR = $30$dB}
  \label{fig:6}
\end{figure}

\subsection{The effect of the phase shift}
\label{sec:effect-phase-shift}

We explore the effect of the phase shift $d$ added to the original
signal on the performance of the phase retrieval. Figure~\ref{fig:7} shows that
$d=4$ is a good choice of the phase shift. The performance is not
sensitive to the $d\geq 4$. And the smaller the $d$ is, the more
iterations are needed for the same relative error level.

\begin{figure}
   \begin{subfigure}[b]{.5\textwidth}
\centering
\caption{}
\includegraphics{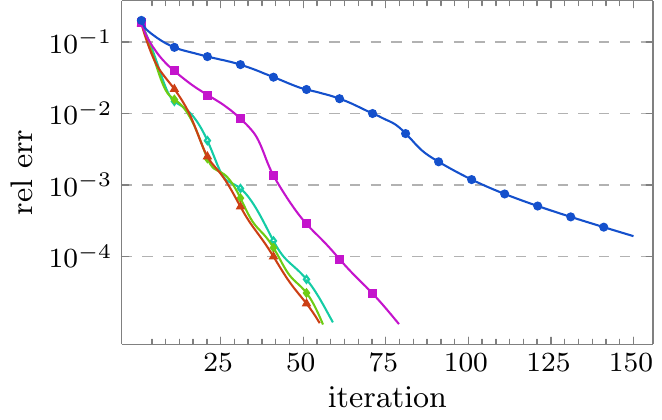}
   \end{subfigure}\hskip 30pt
   \begin{subfigure}[b]{.5\textwidth}
\centering
\caption{}
\includegraphics{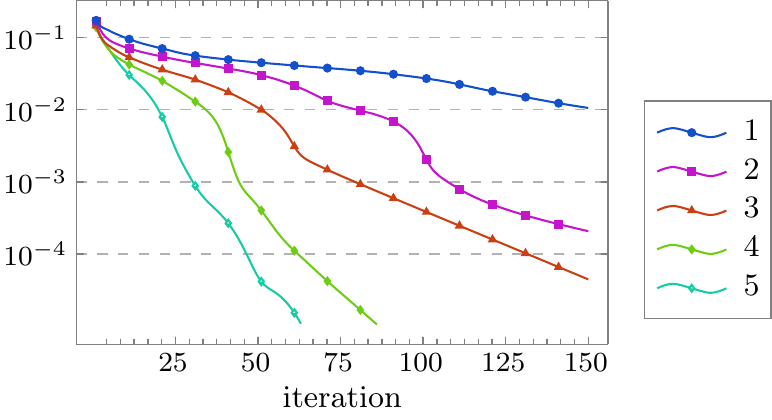}
   \end{subfigure}\\\medskip
   \begin{subfigure}[b]{.5\textwidth}
\centering
\caption{}
\includegraphics{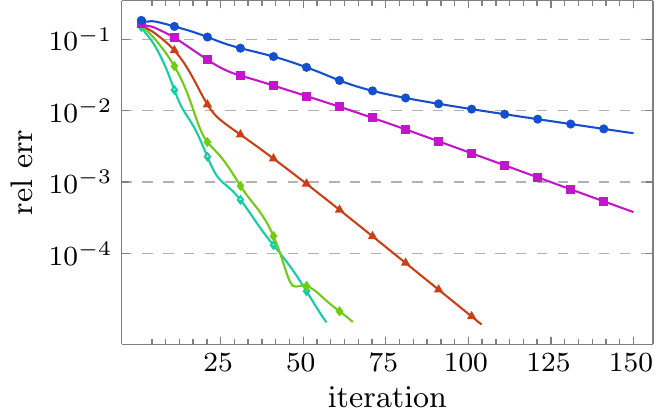}
   \end{subfigure} \hskip 30pt
     \begin{subfigure}[b]{.5\textwidth}
\centering
\caption{}
\includegraphics{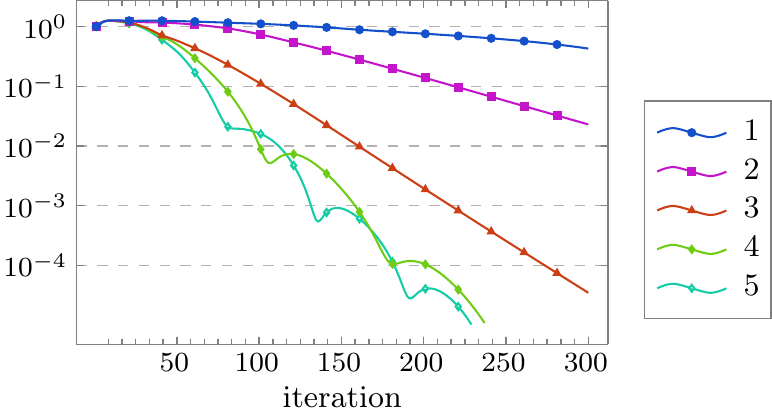}
   \end{subfigure}
  \caption{Relative error vs phase shift $d$ for noiseless case, for
    golden balls the two phase shift are $\pm d$}
 \label{fig:7}
 \end{figure}

\section{Conclusion}
\label{sec:conclusion}
The
knowledge of the Fourier transform intensity from structured illuminations with random
phase masks specifies the signal up
to a single complex constant~\cite{Fannjiang2012a}. However, the random phase mask is
difficult to realize in practice, we propose the structured
illumination with a pixel-dependent deterministic phase shift, which can be
implemented by recording the diffraction pattern in the Fresnel
regime. We extend the RAAR algorithm
for two and more diffraction patterns, which operates in Fourier
domain rather than in space
domain. We follow
the methodology in~\cite{Chen2015} and prove the local 
convergence of the RAAR algorithm. And the ER and HIO are then two special
cases with the relaxation parameter $\beta$ being $0.5$ and $1$. We
found that the iterations of HIO algorithm oscillate 
in the neighborhood of the solution for noiseless case and wander
away from the neighborhood of the solution for noisy data. From the simulations, we recommend
$\beta=0.8$ is applicable for real images and $\beta=0.9$
for complex images. The linear relation between the relative error and
the noise level shows that phase retrieval with two patterns for phase 
retrieval is stable.
\section*{Acknowledgments}
The authors thank Chao Wang for helpful comments and suggestions
on this manuscript draft. The authors are indebted to Stefano
Marchesini for providing us with the gold balls data set used in
numerical simulations. This work was supported by NSF grants of China (61421062, 11471024).

\printbibliography
\end{document}